\documentclass[a4paper]{article}

\usepackage{
amsmath,
amsthm,
amscd,
amssymb,
}
\usepackage{comment}
\usepackage{tikz}
\usetikzlibrary{positioning,arrows,calc}
\usepackage{xspace}

\usepackage{graphicx}

\usepackage{url}

\theoremstyle{plain}
\newtheorem{lemma}{Lemma}
\newtheorem{definition}{Definition}

\newtheorem{proposition}{Proposition}
\newtheorem{theorem}{Theorem}

\newtheorem{remark}{Remark}

\newtheoremstyle{derp}
{3pt}
{3pt}
{}
{}
{\upshape}
{:}
{.5em}
{}
\theoremstyle{derp}

\newcommand{\Z}{\mathbb{Z}}

\title{When are group shifts of finite type?}

\author{
Ville Salo \\
vosalo@utu.fi
}

\begin{document}
\maketitle

\begin{abstract}
It is known that a group shift on a polycyclic group is necessarily of finite type. We show that, for trivial reasons, if a group does not satisfy the maximal condition on subgroups, then it admits non-SFT abelian group shifts. In particular, we show that if a group is elementarily amenable or satisfies the Tits alternative, then it is virtually polycyclic if and only if all its group shifts are of finite type. Our theorems are minor elaborations of results of Schmidt and Osin. 
\end{abstract}

\section{Introduction}


In \cite{Sc95}, it is shown that on polycyclic-by-finite groups, all group shifts are of finite type. We prove a partial converse. A group \emph{satisfies the Tits alternative} if every finitely-generated subgroup either contains a free group on two generators or is virtually solvable. Finite-dimensional linear groups over any field satisfy this. We say $G$ satisfies the \emph{weak Tits alternative} if all its finitely-generated subgroups either contain a free group on two generators or are elementarily amenable.

\begin{theorem}
\label{thm:WeakTits}
Let $G$ be a group that satisfies the weak Tits alternative. Then all group shifts on $G$ are of finite type if and only if $G$ is virtually polycyclic.
\end{theorem}




The theorem is a rather direct corollary of the following theorem, and a characterization of virtually polycyclic groups due to Osin \cite{Os02}. 

\begin{theorem}
\label{thm:Max}
Let $G$ be a group which does not satisfy the maximal condition on subgroups. Then there is a group shift on $G$ which is not of finite type.
\end{theorem}



We give a direct proof of this. The essence of it, and probably also the construction itself, can be extracted from Schmidt \cite{Sc95}. 

The full characterization of groups where all group shifts are of finite type stays open. Theorem~\ref{thm:WeakTits} shows that if such a group is not virtually polycyclic, it is a counterexample to a strong form of the ``Von Neumann conjecture''. Of course there are such examples, such as Tarski monsters (see \cite{Ol82}), which satisfy the maximal condition on subgroups but are not virtually polycyclic. We do not know whether such groups support group shifts that are not of finite type. 




\section{The construction}

In this section, we prove Theorem~\ref{thm:Max}. 

\begin{remark}
The construction is written straight from the definitions, in somewhat excessive detail. This is because we found it (and especially the proof method) quite surprising at the time of writing. Somehow, we had a hard time believing that finite subshifts on non-finitely-generated groups could be essential for obtaining results about uncountable subshifts on finitely-generated groups, and that there had to be some gap in the logic. Of course, after the fact this is nothing but an instance of the observation that sometimes you need to step into a larger family of mathematical objects to understand the smaller family you are actually studying.
\end{remark}

All groups and finite sets are discrete and $A^B$ with $A, B$ discrete has the product topology. The reader may safely assume all groups are countable (the case of uncountable groups is somewhat trivial), but we emphasize that groups are not assumed to be finitely-generated. 

\begin{definition}
Let $G$ be a group and $\Sigma$ a finite set. Then $X \subset \Sigma^G$ is a \emph{subshift} if it is topologically closed and closed under the action of $G$ given by the shifts defined by $(g \cdot x)_h = x_{g^{-1}h}$. A subshift is \emph{of finite type} if there is a clopen subset $C \subset \Sigma^G$ such that $x \in \Sigma^G$ is in $X$ if and only if $\forall g \in G: g \cdot x \in C$. More generally, we say $G \curvearrowright X$ is a subshift if it is \emph{conjugate} to one, that is, homeomorphic to one by a shift-commuting homeomorphism.
\end{definition}

Clearly being of finite type is preserved under conjugacy, since this condition is defined in terms of the topology and the shift maps, as this structure is preserved by conjugacies. For groups that are not finitely-generated, it is not clear that this is a particularly interesting definition, but it corresponds to the usual notion in the finitely-generated case, and non-finitely-generated groups are crucial in the proof even if we are only interested in finitely-generated ones.


\begin{definition}
If $\Sigma$ is a group, then a \emph{group shift} (on alphabet $\Sigma$) is a subshift $X \subset \Sigma^G$ such that $X$ is also a subgroup of $\Sigma^G$ under cellwise operations $(x \cdot y)_g = x_g \cdot y_g$.
\end{definition}

It is known that group shifts as defined above are precisely the internal groups of the category of subshifts (up to operation-preserving isomorphism).

\begin{lemma}
\label{lem:NonSFTGroupShift}
Suppose $G$ is not finitely-generated. Then its trivial action on $\Z/2\Z$ is a group shift, but is not of finite type.
\end{lemma}

\begin{proof}
Let $\Sigma = \Z/2\Z$ and define $X = \{0^G, 1^G\}$. Then clearly $X$ is a subshift of $\Sigma^G$ and $a \mapsto (g \mapsto a)$ defines a conjugacy from the trivial action of $G$ on $\Z/2\Z$ to the subshift $X$, where $X$ inherits a symbolwise group operation from $\Z/2\Z$.

We recall the characterization of a clopen set in $\Sigma^G$ in terms of coordinates: $C \subset \Sigma^G$ is clopen if and only if there exists a finite set $A \subset G$ called the \emph{support} of $C$ such that $x \in C \iff \exists x' \in C: x'|_A = x|_A$.

Suppose now that $X$ is of finite type, and let $C$ be a clopen set defining it with support $A$. Then $A$ generates a proper subgroup of $X$, say $H \leq G$. Define $x \in \{0,1\}^\Sigma$ by $x_g = 1 \iff g \in H$. Then $g \cdot x \in C$ for all $x$, since if $g \in H$ then $(g \cdot x)_h = x_{g^{-1} h} = 1$ for all $h \in H$, so in particular $(g \cdot x)|_A = 1^G|_A$ since $A \subset H$. On the other hand if $g \notin H$, then $(g \cdot x)_h = x_{g^{-1} h} = 0$ for all $h \in H$, so in particular $(g \cdot x)|_A = 0^G|_A$ since $A \subset H$.

This contradicts the assumption that $C$ defines the subshift $X$, since the point $x$ is not in $X$.
\end{proof}

\begin{definition}
Let $H \leq G$ and let $Y \subset \Sigma^H$ be an $H$-subshift. The \emph{coset extension} $Y^{G/H}$ is defined as the subset of $\Sigma^G$ defined by
\[ x \in Y^{G/H} \iff \forall g \in G: g \cdot x|_H \in Y. \]
\end{definition}

Note that this is the ``full'' coset extension, where each coset has an independent copy of $Y$. 


\begin{lemma}
\label{lem:SFTIff}
Let $H \leq G$, and let $Y$ be an $H$-subshift. Then the coset extension $X = Y^{G/H}$ is a $G$-subshift, and it is of finite type if and only if $Y$ is.
\end{lemma}

\begin{proof}
Since $Y$ is closed and translations are continuous, also $X$ is closed. Also $X$ is indeed a subshift: If $x \in X$ and $g' \in G$, then
\[ \forall g \in G: (g \cdot (g' \cdot x))|_H = (gg' \cdot x)|_H \in Y, \]
which implies that $g' \cdot x \in X$.

Let $\pi : \Sigma^G \to \Sigma^H$ be the projection map $\pi(x) = x|_H$. This map is clearly continuous. Suppose that $Y$ is of finite type. Let $D \subset \Sigma^H$ be the clopen set defining $Y$. Define $C \subset \Sigma^G$ by
\[ x \in C \iff \pi(x) = x|_H \in D, \]
that is, $C = \pi^{-1}(D)$. Then $C$ and $X \setminus C$ are open as they are preimages of the open sets $D$ and $Y \setminus D$, so $C$ is clopen.

We claim that $x \in \Sigma^G$ is in $X$ if and only if $g \cdot x \in C$ for all $g \in G$. To see this, suppose first that $x \in X$. Then for all $g \in G$, $g \cdot x|_H \in Y$, so in particular $g \cdot x|_H \in D$, implying $g \cdot x \in C$. Next, suppose $g \cdot x \in C$ for all $g \in G$. Then in particular $h \cdot (g \cdot x) \in C$ for all $g \in G, h \in H$, and thus $(h \cdot (g \cdot x))|_H \in D$. Let $y = (g \cdot x)|_H$. Then $(h \cdot (g \cdot x))|_H = h \cdot y$ because
\[ (h \cdot (g \cdot x))_{h'} = (g \cdot x)_{h^{-1} h'} = y_{h^{-1} h'} = (h \cdot y)_{h'}. \]
It follows that $h \cdot y \in D$ for all $h \in H$, so $y \in Y$. This shows that $g \cdot x|_H \in Y$ for all $g \in G$, implying $x \in X$. We have shown that $X$ is of finite type.



Next, we prove that if $X$ is SFT, then so is $Y$. Since even sets of $\Z$-rows of $\Z^2$-SFTs can be $\Pi^0_1$-hard, this is obviously very specific to the fact we are restricting precisely the coset extension $X = Y^{G/H}$ of some $H$-subshift $Y$ back to $H$.

Suppose that the clopen set $C$ defining $X$ has support $A$. We claim that there is a clopen set with support $B = A^{-1} A \cap H$ defining $Y$. Suppose not. Then there is in particular a point $y \in \Sigma^H \setminus Y$ such that for any $h \in H$, there exists a point $y^h \in Y$ such that $(h \cdot y^h)|_B = (h \cdot y)|_B$, as otherwise we could pick $\{z \;|\; z|_B \in Y|_B\}$ as a clopen set with support $B$ defining $Y$.

If $Y$ is empty, then it is an SFT (since the empty set is clopen) and we are done. Otherwise pick any $z \in Y$, and let $y \notin Y$ and $y^h \in Y$ for all $h \in H$ be as above. Pick representatives $1_G = g_1, g_2, \ldots$ for the left cosets of $H$, so that $g_i H$ and $H^{-1} g_i^{-1} = H g_i^{-1}$ form two (possibly distinct) partitions of $G$, and $H = g_1 H = H g_1^{-1}$. Define $x \in \Sigma^G$ by $x_{h} = y_h$ and $x_{g_i h} = z_h$ whenever $h \in H$ and $i > 1$, and similarly define $x^h$ for all $h \in H$ by $x^h_{h'} = y^h_{h'}$ and $x^h_{g_i h'} = z_{h'}$ whenever $h' \in H$ and $i > 1$.

We claim that $x^h \in X$ for all $h \in H$. For this let $g \in G$ be arbitrary, and write $g = \hat h g_i^{-1}$. If $i > 1$, then $g \notin H$ and we have
\[ (g \cdot x^h)_{h'} = x^h_{g^{-1} h'} = x^h_{g_i \hat h^{-1} h'} = z_{\hat h^{-1} h'} = (\hat h \cdot z)_{h'} \]
so in this case $g \cdot x^h|_H = \hat h \cdot z \in Y$. Otherwise
\[ (g \cdot x^h)_{h'} = x^h_{g^{-1} h'} = x^h_{g_i \hat h^{-1} h'} = y^h_{\hat h^{-1} h'} = (\hat h \cdot y^h)_{h'} \]
so again $g \cdot x^h|_H = \hat h \cdot y^h \in Y$. By the definition of the coset extension, we have $x^h \in X$.

We claim that now also $x \in X$. Suppose not, so that $g \cdot x \notin C$ for some $g \in G$. First, suppose that $g^{-1} A$ does not intersect $H$, pick $h \in H$ arbitrarily and compute
\[ (g \cdot x^h)_a = x^h_{g^{-1} a} = x_{g^{-1} a} = (g \cdot x)_a \]
for all $a \in A$, so actually $g \cdot x^h \notin C$, a contradiction.

Suppose next that $g^{-1} A$ does intersect $H$, and let $g^{-1} a = h^{-1} \in H$. We claim that then $g \cdot x^h \notin C$. Let $a' \in A$. If $g^{-1} a' \notin H$, then as above
\[ (g \cdot x^h)_{a'} = x^h_{g^{-1} a'} = x_{g^{-1} a'} = (g \cdot x)_{a'}. \]
If $g^{-1} a' \in H$, then also $b = (g^{-1} a)^{-1} g^{-1} a' = a^{-1} a' \in H$. Note also that, by form, we have $b \in B$. We thus have
\[ (g \cdot x^h)_{a'} = x^h_{g^{-1} a'} = y^h_{g^{-1} a'} = y^h_{g^{-1} a b} \]
where the second equality follows because $g^{-1} a' \in H$, and
\[ (g \cdot x)_{a'} = x_{g^{-1} a'} = y_{g^{-1} a'} = y_{g^{-1} a b} \]
by the same computation.  Now $y^h_{g^{-1} a b} = y_{g^{-1} a b}$ because $b \in B$ and by the choice of the $y^h$:
\[ y^h_{g^{-1} a b} = y^h_{h^{-1} b} = (h \cdot y^h)_b = (h \cdot y)_b = y_{h^{-1} b} = y_{g^{-1} a b}. \]
We have shown that in any case $(g \cdot x^h)_{a'} = (g \cdot x)_{a'}$ for $a' \in A$, which implies $g \cdot x^h \notin C$, a contradiction with $x^h \in X$.
\end{proof}

\begin{lemma}
\label{lem:CosetExtGroupShift}
If $\Sigma$ is a finite group, $H \leq G$ and $Y \subset \Sigma^H$ is a group shift, then $Y^{G/H}$ is a group shift.
\end{lemma}

\begin{proof}
If $x, x' \in Y^{G/H}$, $g \in G$ and $h \in H$, then
\[ (g \cdot (x \cdot x'))_h = (x \cdot x')_{g^{-1} h} = x_{g^{-1} h} \cdot x'_{g^{-1} h} = (g \cdot x)_h \cdot  (g \cdot x')_h. \]
Letting $g \cdot x = y$, $g \cdot x' = y'$, we have $y \cdot y' \in Y$, implying $(g \cdot (x \cdot x')) \in Y$. The proof for inverses is similar.
\end{proof}

\begin{definition}
A group $G$ \emph{satisfies the maximal condition on subgroups} if for any sequence $G_1 \leq G_2 \leq G_3 \leq ...$ of subgroups, the sequence eventually stabilizes, i.e. for some $i$, $G_i = G_{i+j}$ for all $j \geq 0$.
\end{definition}

\begin{lemma}
A group $G$ satisfies the maximal condition on subgroups if and only if all its subgroups are finitely-generated.
\end{lemma}

\begin{proof}
If $H \leq G$ is not finitely-generated, then clearly one can enumerate a countable sequence of increasing subgroups of $H$, and in particular they are subgroups of $G$. If $G$ does not satisfy the maximal condition on subgroups, then $\bigcup_i G_i$ is a subgroup that is obviously not finitely-generated.
\end{proof}

\begin{proof}[Proof of Theorem~\ref{thm:Max}]
Suppose $H \leq G$ is not finitely-generated. Then by Lemma~\ref{lem:NonSFTGroupShift} there is a group shift $Y \subset \Sigma^H$ which is not of finite type. Then by Lemma~\ref{lem:CosetExtGroupShift}, $Y^{G/H}$ is a group shift, and by Lemma~\ref{lem:SFTIff} it is not of finite type.
\end{proof}


\section{Groups without the weak Markov property}

\begin{definition}
A group $G$ has the \emph{weak Markov property} if all group shifts it supports are of finite type.
\end{definition}

This is a weaker version of the notion of \emph{Markov type} as defined in \cite{Sc95}, see Lemma~\ref{lem:MarkovGSFT} for a proof. Theorem~\ref{thm:Max} states precisely that groups having the weak Markov property satisfy the maximal condition on subgroups.

The maximal condition on subgroups is very strong. We include a brief discussion of groups that do not have this property, and in particular prove Theorem~\ref{thm:WeakTits}.

We begin with a few examples which cover some of our favorite groups such as free groups and the lamplighter group.

\begin{proposition}
A free product $G * H$ for nontrivial $G, H$ satisfies the maximal condition on subgroups if and only if it has the weak Markov property if and only if $|G| = |H| = 2$.
\end{proposition}

\begin{proof}
The case $|G| = |H| = 2$ corresponds to the infinite dihedral group, which is polycyclic and thus covered by Proposition~\ref{prop:VirtSol} below. When one of the groups has three generators, we pick $g \in G \setminus \{1_G\}$ and $h, h' \in H \setminus \{1_H\}$ with $h \neq h'$, and then $gh, gh'$ are easily seen to generate a free group with two generators using the normal form theorem of free products \cite{LySc15}. The free group with two generators contains a copy of the free group with infinitely many generators, and thus does not satisfy the maximal condition on subgroups, and we apply Theorem~\ref{thm:Max}.
\end{proof}

\begin{proposition}
A wreath product $G \wr H$ with $G$ nontrivial and $H$ infinite never satisfies the maximal condition on subgroups, thus never has the weak Markov property.
\end{proposition}

\begin{proof}
When $H$ is infinite, $G^\omega$ (the direct union of $G^d$ as $d \rightarrow \infty$) is contained in $G \wr H$, as the conjugates $hGh^{-1}$ commute and have trivial intersection for distinct $h \in H$. The group $G^\omega$ is obviously never finitely generated when $G$ is non-trivial, and we apply Theorem~\ref{thm:Max}.
\end{proof}

\begin{proposition}
If $G$ is uncountable, then it does not have the weak Markov property.
\end{proposition}

\begin{proof}
Pick a countable sequence $g_1, g_2, g_3, ...$ such that $g_{i+1} \notin \langle g_1, g_2..., g_i \rangle$ for all $i$, which can be done since finitely-generated subgroups are countable. Then $\langle g_1, g_2, \ldots \rangle$ is not finitely-generated since it admits an infinite generating set having no finite generating subset, so this follows from Theorem~\ref{thm:Max}.
\end{proof}

For the main theorem, we will use a theorem of Osin, but while covering the virtually polycyclic case, we give a self-contained proof of the virtually solvable case in Proposition~\ref{prop:VirtSol}. 
The following is classical, see Theorem~5.4.12 of \cite{Ro96}.

\begin{lemma}
A solvable group is polycyclic if and only if it satisfies the maximal condition on subgroups.
\end{lemma}

\begin{lemma}
\label{lem:VirtualVersion}
A virtually solvable group is virtually polycyclic if and only if it satisfies the maximal condition on subgroups.
\end{lemma}

\begin{proof}
Suppose $G$ is virtually solvable. Suppose first it is virtually polycyclic, and let $H \leq G$ be a finite-index polycyclic subgroup. If $K \leq G$, then $[K : K \cap H] \leq [G : H] < \infty$. Since $K \cap H$ is finitely generated as a subgroup of $H$, so is $K$, by adding finitely many coset representatives to a generating set of $K \cap H$. Thus every subgroup of $G$ is finitely generated.

Suppose then that $G$ satisfies the maximal condition on subgroups, and let $K \leq G$ be a solvable group of finite index. It is enough to show that $K$ is polycyclic. Suppose it is not. Then there is a subgroup $H \leq K$ which is not finitely generated, and then also $H \leq G$, a contradiction.
\end{proof}

The following is Definition~4.1 in \cite{Sc95}.

\begin{definition}
A countable group $G$ is of \emph{Markov type} if, for every compact Lie group $\Sigma$, the shift-action of $G$ on $\Sigma^G$ satisfies the descending chain condition.
\end{definition}

In \cite{Sc95}, Lie groups are defined to be compact matrix groups over the complex numbers. In particular, every finite group with the discrete topology is a Lie group. For the argument, we do not need to know what the descending chain condition, or \emph{d.c.c.}, means, but it refers to that for the partially ordered set of shift-invariant closed subgroups. The following is Theorem~4.2 of \cite{Sc95}.

\begin{theorem}
\label{thm:PolyMarkov}
If $G$ is polycyclic-by-finite, then it is of Markov type.
\end{theorem}

The following is Theorem~3.8 of \cite{Sc95}.

\begin{theorem}
\label{thm:DCCSFT}
Let $G$ be a countable group, and let $\Sigma$ be a compact Lie group. The shift action of $G$ on $\Sigma^G$ satisfies the d.c.c. if and only if every closed shift-invariant subgroup of $\Sigma^G$ is of finite type.
\end{theorem}

This definition of finite type restricts to ours in the zero-dimensional setting, and we omit the general definition.

\begin{lemma}
\label{lem:MarkovGSFT}
Let $G$ be a countable group. If it is of Markov type, it has the weak Markov property.
\end{lemma}

\begin{proof}
Let $X \subset \Sigma^G$ be a group shift. Then $\Sigma$ can be represented by permutation matrices, and thus is a compact Lie group in the sense of \cite{Sc95}. If $G$ has the Markov property, then $\Sigma^G$ satisfies the d.c.c. by definition, and thus $X$ is of finite type by Theorem~\ref{thm:DCCSFT}. This proves the weak Markov property.
\end{proof}

\begin{proposition}
\label{prop:VirtSol}
Let $G$ be virtually solvable. Then $G$ has the weak Markov property if and only if it is virtually polycyclic.
\end{proposition}

\begin{proof}
If $G$ is virtually polycyclic, then it is in particular polycyclic-by-finite (by intersecting all conjugates of the polycyclic subgroup of finite index), so by Theorem~\ref{thm:PolyMarkov} it has the Markov property, so by Lemma~\ref{lem:MarkovGSFT} it has the weak Markov property. If $G$ is virtually solvable but not virtually polycyclic, then by Lemma~\ref{lem:VirtualVersion} it does not satisfy the maximal condition on subgroups, and thus does not have the weak Markov property by Theorem~\ref{thm:Max}.
\end{proof}


We now prove Theorem~\ref{thm:WeakTits} (in slightly stronger form). The following follows by combining Proposition~4.5 and Example 4.8 (3) of \cite{Sc95}, and is the Markov type version of our construction in Theorem~\ref{thm:Max}.

\begin{lemma}
If $G$ is of Markov type and $H \leq G$, then $H$ is of Markov type. If a group is of Markov type, then it satisfies the maximal condition on subgroups.
\end{lemma}


\begin{definition}
The class of \emph{elementary amenable groups} is the smallest class containing all finite and abelian groups, which is closed under subgroups, quotients, extensions and directed unions.
\end{definition}

The following is Theorem~2.1 in \cite{Os02}.

\begin{lemma}
\label{lem:EAGMAX}
An elementary amenable group satisfies the maximal condition on subgroups if and only if it is virtually polycyclic.
\end{lemma}

\begin{theorem}
Let $G$ be a group such that every finitely-generated subgroup of it either contains a free group on two generators or is elementarily amenable. Then the following are equivalent:
\begin{itemize}
\item $G$ satisfies the maximal condition on subgroups,
\item $G$ is virtually polycyclic,
\item $G$ has the weak Markov property,
\item $G$ is of Markov type.
\end{itemize}
\end{theorem}

\begin{proof}
Since virtual polycyclicity proves the weak Markov property and Markov type, and not having the maximal condition on subgroups proves the group does not have the weak Markov property and is not of Markov type, it is enough to show that $G$ is virtually polycyclic if and only if it satisfies the maximal condition on subgroups.

Let $H \leq G$ be finitely-generated. If $H$ contains a two-generator free group, then $G$ contains an infinitely-generated free group, and thus does not satisfy the maximal condition on subgroups.

Suppose then that $H$ does not contain a two-generator free group. Then it is elementary amenable by assumption. It is then enough to show that an elementary amenable group satisfies the maximal condition on subgroups if and only if it is virtually polycyclic. This is Lemma~\ref{lem:EAGMAX}.
\end{proof}

\bibliographystyle{plain}
\bibliography{../../../bib/bib}{}

\def\ocirc#1{\ifmmode\setbox0=\hbox{$#1$}\dimen0=\ht0 \advance\dimen0
  by1pt\rlap{\hbox to\wd0{\hss\raise\dimen0
  \hbox{\hskip.2em$\scriptscriptstyle\circ$}\hss}}#1\else {\accent"17
  #1}\fi}\def\cprime{$'$}
\begin{thebibliography}{1}

\bibitem{LySc15}
R.C. Lyndon and P.E. Schupp.
\newblock {\em Combinatorial Group Theory}.
\newblock Classics in Mathematics. Springer Berlin Heidelberg, 2015.

\bibitem{Ol82}
A.~Yu Ol'shanskii.
\newblock Groups of bounded period with subgroups of prime order.
\newblock {\em Algebra and Logic}, 21(5):369--418, Sep 1982.

\bibitem{Os02}
D.~V. Osin.
\newblock Elementary classes of groups.
\newblock {\em Mathematical Notes}, 72(1):75--82, Jul 2002.

\bibitem{Ro96}
D.~Robinson.
\newblock {\em A Course in the Theory of Groups}.
\newblock Graduate Texts in Mathematics. Springer New York, 1996.

\bibitem{Sc95}
Klaus Schmidt.
\newblock {\em Dynamical systems of algebraic origin}, volume 128 of {\em
  Progress in Mathematics}.
\newblock Birkh\"auser Verlag, Basel, 1995.

\end{thebibliography}

\end{document}